\documentclass[11pt]{amsart}
\usepackage{amsmath,amsfonts,amssymb}
\usepackage{color,graphicx}
\usepackage[left=2.5cm,right=2.5cm,top=3.5cm,bottom=3cm]{geometry}
\usepackage{algorithm}
\usepackage{algorithmic}

\newtheorem{thm}{Theorem}[section]
\newtheorem{lem}[thm]{Lemma}
\newtheorem{cor}[thm]{Corollary}
\newtheorem{prop}[thm]{Proposition}
\theoremstyle{definition}
\newtheorem{rem}[thm]{Remark}
\newtheorem{example}[thm]{Example}

\def\N{\mathbb{N}}
\def\R{\mathbb{R}}

\def\e{\varepsilon}

\def\nn{|\!|\!|}

\begin{document}

\title{Forward-backward approximation of evolution equations in finite and infinite horizon}

\author{Andr\'es Contreras}
\address{Departamento de Ingenier\'\i a Matem\'atica \& Centro de Modelamiento Matem\'atico (CNRS UMI2807), FCFM, Universidad de Chile, Beauchef 851, Santiago, Chile}
\email{acontreras@dim.uchile.cl}

\author{Juan Peypouquet}
\address{Departamento de Ingenier\'\i a Matem\'atica \& Centro de Modelamiento Matem\'atico (CNRS UMI2807), FCFM, Universidad de Chile, Beauchef 851, Santiago, Chile}
\email{jpeypou@dim.uchile.cl}

\thanks{Supported by FONDECYT Grant 1181179, CMM-Conicyt PIA AFB170001, ECOS-CONICYT Grant C18E04 and CONICYT-PFCHA/DOCTORADO NACIONAL/2016 21160994}


\maketitle

\begin{abstract}
This research is concerned with evolution equations and their forward-backward discretizations. Our first contribution is an estimation for the distance between iterates of sequences generated by forward-backward schemes, useful in the convergence and robustness analysis of iterative algorithms of widespread use in variational analysis and optimization. Our second contribution is the approximation, on a bounded time frame, of the solutions of evolution equations governed by accretive (monotone) operators with an additive structure, by trajectories defined using forward-backward sequences. This provides a short, simple and self-contained proof of existence and regularity for such solutions; unifies and extends a number of classical results; and offers a guide for the development of numerical methods. Finally, our third contribution is a mathematical methodology that allows us to deduce the behavior, as the number of iterations tends to $+\infty$, of sequences generated by forward-backward algorithms, based solely on the knowledge of the behavior, as time goes to $+\infty$, of the solutions of differential inclusions, and viceversa.
\end{abstract}


\section{Introduction}

Semigroup theory is a relevant tool in the study of ordinary and partial differential equations, as well as differential inclusions, which appear, for instance, in contact mechanics, optimization, variational analysis and game theory. Among its applications, it helps analyze the evolution of flows in mechanical systems, and establish convergence and convergence rates for numerical optimization algorithms. One of its cornerstones was the Hille-Yosida Theorem \cite{Hil,Yos}, which states that an {\it unbounded} linear operator $A$, on a Banach space $X$, is the infinitessimal generator of a strongly continuous semigroup $(\mathcal S_t)_{t\ge 0}$ of nonexpansive linear operators on $X$, satisfying $-\dot u(t)=Au(t)$ if, and only if, it is closed, its domain is dense in $X$, its spectrum does not intersect $\R_-$, and the resolvents satisfy an appropriate bound. This result was complemented by the Lumer-Phillips Theorem \cite{Phi,LumPhi}, which provides an alternative, and, perhaps more practical, characterization in terms of semidefiniteness. It is important to mention that Hille and Yosida used different strategies to construct the semigroup (that is, to show the necessity). Yosida's approach consists in approximating the operator $A$ by a family $(A_\lambda)_{\lambda>0}$ of {\it bounded} ones, establishing the existence of solution to the regularized differential equation $-\dot u_\lambda(t)=A_\lambda u_\lambda(t)$ by classical arguments, and then passing to the limit while showing that the regularized solutions $u_\lambda$ converge to a true solution of the original problem. Hille, in turn, discretizes the time interval $[0,T]$, where $T>0$ is arbitrary but fixed, constructs approximating trajectories using a sequence of points generated by resolvent iterations, and finally passes to the limit as the partition is refined. Both show the convergence is uniform on $[0,T]$.

Another important landmark was the discovery, two decades later, of sufficient conditions for a nonlinear, possibly multi-valued, operator $A$ to generate a strongly continuous semigroup $(\mathcal S_t)_{t\ge 0}$ of nonexpansive nonlinear operators that solves the differential inclusion $-\dot u(t)\in Au(t)$. Yosida's approach was used by Br\'ezis \cite{Bre} (also Barbu \cite{Bar} and Pazy \cite{Paz}), while Hille's path was followed by Crandall and Pazy \cite{CraPaz}\footnote{Although Crandall and Liggett \cite{CraLig} used Yosida's method in their work on Banach spaces.}, and then simplified and perfected by Rasmussen \cite{Ras} and Kobayashi \cite{K}. They built a concise and sharp inequality $-$let us call it (I)$-$ to bound the distance between two sequences of points generated using compositions of resolvents. We shall come back to this point later, since this is the line of research we explore in this paper. Other authors have analyzed the nonautonomous setting \cite{KKO,AlvPey1,BiaHac}, where there is a function $t\mapsto A(t)$ that generates an {\it evolution system} that, of course, is not a semigroup, in general. In some relevant special cases, resolvents may be replaced by Krasnosel'ski\u\i-Mann \cite{Kra,Man} and, equivalently, Euler \cite{PeySor} iterations. This issue is addressed in \cite{Vig,ConPey1}, where applications in optimization and game theory are given. 

A few years later, Passty \cite{P} introduced the notion of an {\it asymptotic semigroup}, which is, roughly speaking, a possibly nonautonomous evolution system that asymptotically behaves like a semigroup. This concept allows us to deduce several convergence properties of the trajectories generated by an asymptotic semigroup, as time goes to $+\infty$, based on what is known about those generated by the semigroup it is related to. A similar idea lies behind the notion of {\it almost-orbit} (see \cite{MK}), which helps to prove that every nonexpansive iterative algorithm is robust against summable errors (see \cite[Lemma 5.3]{J}). The interested reader is referred to \cite{AlvPey1,AlvPey2,AlvPey3} for further details and applications. Passty proved, under some restrictive assumptions, that every sequence generated using products of resolvents of $A$, with parameters $(\lambda_n)_{n\ge 0}$, more precisely, satisfying $x_n=(I+\lambda_nA)^{-1}x_{n-1}$ for all $n$, converges strongly (weakly) as $n\to+\infty$ if, and only if, all trajectories generated by the semigroup $(\mathcal S_t)_{t\ge 0}$ converge strongly (weakly) as $t\to+\infty$. The process of generating sequences of points using resolvent iterations is also known as the {\it proximal point algorithm}, as developed by Martinet \cite{Mar} and further studied by Rockafellar \cite{Roc} and Br\'ezis-Lions \cite{BreLio}, among others. It is one of the fundamental building blocks of first order methods used to solve nonsmooth optimization problems and variational inequalities in practice (see the note on forward-backward iterations in the next paragraph). Passty's innovative idea is remarkable, since it makes it possible to use calculus techniques, such as derivation and integration, to analyze the behavior of iterative algorithms. A few years later, Miyadera and Kobayashi \cite{MK} and Sugimoto and Koizumi \cite{SugKoi} were able to get rid of Passty's superfluous hypotheses by using inequality (I) mentioned above. Inequality (I) also enabled G\"uler \cite{Gul} to show, based on an example of Baillon \cite{Bai}, that there is a proper, lower-semicontinuous, convex function for which the proximal point algorithm produces sequences that converge weakly but not strongly, settling an open question in optimization theory posed by Rockafellar \cite{Roc} fifteen years earlier. As a matter of fact, this function may be chosen differentiable and with Lipschitz-continuous gradient, as proved by the authors in \cite{ConPey1}, using a variant of inequality (I).

{\it Forward-backward} iterations combine the principles of proximal, Krasnosel'ski\u\i-Mann and Euler iterations. They are fundamental in the numerical analysis of structured optimization problems and variational inequalities, since they represent the core of first order methods. Particular cases include: the gradient method, originally introduced by Cauchy in \cite{Cau}; its variant, the projected gradient method \cite{Gol,LevPol}; the proximal point algorithm mentioned above; the proximal-gradient algorithm \cite{P,LioMer}, and its particular instance, ISTA\footnote{Iterative Shrinkage Thresholding Algorithm.} \cite{DauDefDem,ComWaj}, with applications in image and signal processing, data analysis and machine learning. Moreover, some primal dual methods \cite{ChaPoc,Con,Vu} can be reduced to these types of iterations. Also, accelerated methods, such as FISTA\footnote{Fast Iterative Shrinkage Thresholding Algorithm.} \cite{Nes,BecTeb} use a forward-backward engine.

The purpose of this research is to extend, unify and condense the theory on the generation of strongly continuous semigroups of nonlinear and nonexpansive mappings by multi-valued operators with an additive structure. On the one hand, we analyze the approximation of solutions for the differential inclusion $-\dot u(t)\in (A+B)u(t)$ by trajectories constructed by interpolation of sequences generated using forward-backward iterations, on a compact time interval. This approach is different from the one by Trotter \cite{Tro} and Kato \cite{Kat2}, which uses {\it double backward} iterations. Double backward iterations require the (costly!) computation of both resolvents. We address this issue, for theoretical curiosity, in a forthcoming paper. On the other hand, we establish asymptotic equivalence results that link the behavior, as the number of iterations tends to $+\infty$, of sequences generated by forward backward iterations, with the behavior of the solutions of the differential inclusion $-\dot u(t)\in (A+B)u(t)$, as time $t$ tends to $+\infty$. We obtain new strong convergence results for forward-backward sequences as straightforward corollaries. We have aimed at presenting these findings in a simple and pedagogic manner, accessible to researchers in functional analysis, differential equations and optimization.

Although the Hilbert space setting is suitable for many applications, our results may be stated and proved in a class of Banach spaces with no additional effort. The extension to general Banach spaces is an open question.

The paper is organized as follows: In Section \ref{S:preliminaries}, we give the notation and definitions, along with a description of the main technical tool required to prove our main results. The approximation in a finite time horizon is discussed in Section \ref{S:finite}. Section \ref{S:infinite_I} is devoted to the approximation in an infinite time horizon and contains new convergence results for forward-backward sequences. The technical proofs are given in Section \ref{S:proof}.

\section{Forward-backward iterations defined by accretive and cocoercive operators} \label{S:preliminaries}

Let $X$ be a Banach space with topological dual $X^{\ast}$. Their norms and the duality product are denoted by $\|\cdot\|$, $\|\cdot\|_{\ast}$ and $\langle\cdot,\cdot\rangle$, respectively. The {\it duality mapping} $j:X\rightarrow X^{\ast}$ is defined by $$j(x)=\{x^{\ast}\in X^{\ast}:\langle x^{\ast},x \rangle =\|x\|^{2}=\|x^{\ast}\|_{\ast}^{2}\}.$$
In what follows, we assume that $X^*$ is 2-uniformly convex, which implies that $X$ is reflexive, the duality mapping is single valued, and there is a constant $\kappa>0$ such that
\begin{equation}\label{eq1}
\|u+v\|^{2}\leq \|u\|^{2}+2\langle j(u),v\rangle +\kappa\|v\|^{2},
\end{equation}
for all $u,v\in X$ (see \cite{LT,Xu}). For instance, $L^p$ spaces have this property for $p\ge 2$.

A set-valued operator $A:X \rightarrow 2^{X}$ is {\it accretive} if, whenever $u\in Ax$ and $v\in Ay$, we have
$$\|x-y+\lambda(u-v)\|\geq \|x-y\|$$
for all $\lambda>0$. If, moreover, $I+\lambda A$ is surjective for all $\lambda>0$, we say $A$ is {\it m-accretive}. In this case, its {\it resolvent}, defined as $J_{\lambda}=(I+\lambda A)^{-1}$, is single-valued, everywhere defined and nonexpansive. It follows from \cite[Lemma 1.1]{Kato} that $A$ is accretive if, and only if, it is {\it monotone}, which means that 
$$\langle j(x-y), u-v\rangle \geq 0,\footnote{In Hilbert spaces, this terminology is preferred, and the inequality reads $(x-y,u-v)\ge 0$, where $(\cdot,\cdot)$ is the inner product.}$$
whenever $u\in Ax$ and $v\in Ay$. Next, an operator $B:X\rightarrow X$ is {\it cocoercive} with parameter $\theta>0$ if 
$$\langle j(x-y),Bx-By \rangle  \geq \theta\|Bx-By\|^{2},$$
for all $x,y\in X$. Clearly, if $B$ is cocoercive with parameter $\theta$, it is Lipschitz-continuous with constant $\frac{1}{\theta}$. Moreover, the operator $E_\lambda:X\to X$, defined by
\begin{equation}\label{E_lambda_Banach}
E_{\lambda}=I-\lambda B,
\end{equation}
is nonexpansive for all $\lambda\in [0,\frac{2\theta}{\kappa}]$. 
Finally, if $A$ is $m$-accretive and $B$ is cocoercive, then $A+B$ is $m$-accretive, and the {\it forward backward splitting operator} 
$T_\lambda:X\to X$, defined by 
$$T_\lambda=J_{\lambda}\circ E_\lambda,$$
is single-valued, everywhere defined and nonexpansive. These are the standing assumptions on $X$, $A$ and $B$ for the rest of the paper.

\begin{rem}
Actually, the minimal hypotheses on $\lambda$, $B$ and $X$, required for our proofs to hold, is that $E_\lambda$ be nonexpansive for all $\lambda\in[0,\Lambda]$ for some $\Lambda>0$. Some definitions and proofs must be slightly adjusted if the duality mapping $j$ is not single-valued. If $B=0$, no assumptions need be made on $X$ or $\lambda$.
\end{rem}

We are interested in the study of sequences satisfying 
\begin{equation} \label{E:def_FBseq}
x_{k}=T_{\lambda_k}(x_{k-1})=J_{\lambda_{k}}\big( E_{\lambda_{k}}(x_{k-1})\big)
\end{equation} 
for $k\ge 1$, where $(\lambda_k)$ is a sequence of positive numbers, called {\it step sizes}, and $x_0\in X$ is the {\it initial point}. We mentioned earlier that these sequences are fundamental in the numerical analysis of optimization problems, variational inequalities and fixed-point problems. However, our purpose here is to analyze them as discrete approximations of an evolution equation governed by the sum $A+B$. To this end, it is useful to rewrite \eqref{E:def_FBseq} as 
\begin{equation}\label{BF_iteration}
-\frac{x_{k}-x_{k-1}}{\lambda_{k}} \in Ax_{k}+Bx_{k-1},\quad k\ge 1,
\end{equation}
or, more generally, as
\begin{equation} \label{E:FB_errors}
-\frac{x_{k}-x_{k-1}}{\lambda_{k}} +\e_k\in Ax_{k}+Bx_{k-1},\quad k\ge 1,
\end{equation}
where $\e_k$ accounts for possible perturbations or computational errors. In the notation of formula \eqref{E:def_FBseq}, this is
\begin{equation} \label{E:def_FBseq_errors}
x_{k}=J_{\lambda_{k}}\big( E_{\lambda_{k}}(x_{k-1})+\lambda_k\e_k\big).
\end{equation} 
Back to the exact version \eqref{BF_iteration}, the left-hand side can be interpreted as a discretization of the velocity for a trajectory $t\mapsto u(t)$, so \eqref{BF_iteration} can be related to the differential inclusion
\begin{equation} \label{E:diff_incl}
-\dot u(t)\in Au(t)+Bu(t),
\end{equation} 
for $t>0$. In the following sections, we shall establish the nature of this relationship. On the one hand, we shall prove that the iterations described in \eqref{BF_iteration} can be used, in at least two different ways, to construct a sequence of curves that approximate the solutions of \eqref{E:diff_incl}
uniformly on each compact time interval. The existence of such solutions is obtained as a byproduct. On the other hand, we shall show that, given $A$ and $B$, the trajectories satisfying \eqref{E:diff_incl} will have the same convergence properties, when $t\to \infty$, as the sequences satisfying \eqref{BF_iteration}, when $k\to \infty$, provided the step sizes are sufficiently small. The key mathematical tool is the following inequality, whose proof is technical, and will be given in Section \ref{S:proof}.

\begin{thm} \label{T:Kobayashi_A+B}
	Let  $(x_{k})$, $(\hat{x}_{l})$ be two sequences generated by \eqref{E:FB_errors}, with step sizes $(\lambda_{k})$ and $(\hat{\lambda}_{l})$, as well as error sequences $(\e_k)$ and $(\hat \e_l)$. Assume $\lambda_k,\hat\lambda_l\le \frac{\theta}{\kappa}$ for all $k,l\in\N$. Then, for $u\in D(A)$ fixed, and each $k,l\in\N$, we have
	\begin{equation}\label{koba_A+B}
	\|x_{k}-\hat{x}_{l}\|\leq\|x_0-u\|+\|\hat{x}_0-u\|+\nn(A+B)u\nn\sqrt{(\sigma_{k}-\hat{\sigma}_{l})^{2}+\tau_{k}+\hat{\tau}_{l}}+e_k+\hat e_l,
	\end{equation}
	where $\displaystyle \nn Au\nn=\inf_{v\in Au}\|v\|$,  $\displaystyle \sigma_{k}=\sum_{i=1}^{k}\lambda_{i}$, $\displaystyle \tau_{k}=\sum_{i=1}^{k}\lambda^{2}_{i}$ and $e_k=\sum_{i=1}^k\lambda_i\|\e_i\|$ (similarly for $\hat{\sigma}_{l}$, $\hat{\tau}_{l}$ and $\hat e_l$).	
\end{thm}

We first became aware of an inequality of this sort (for $B\equiv 0$ and slightly less sharp) in \cite{Gul}, where G\"uler attributes it to Kobayashi \cite{K} (see also \cite{PeySor}). However, the main arguments were given by Rasmussen \cite{Ras}, who simplified the proof of Crandall and Liggett \cite{CraLig}, ultimately based on that of Hille \cite{Hil}. Similar estimations are given in \cite{KKO,AlvPey1} (still for $B=0$, but for a time-dependent $A$) and in \cite{Vig,ConPey1} for $A=0$.

\section{Approximation in finite horizon} \label{S:finite}

Theorem \ref{T:Kobayashi_A+B} provides existence and regularity results for the evolution equation
\begin{equation}\label{eq10}
\left\{%
\begin{array}{rcll}
-\dot u(t) & \in&(A+B)u(t), & \hbox{for almost every }t>0,\\
u(0) & = & u_{0}\in\overline{D(A)},
\end{array}%
\right.
\end{equation}
by means of an approximation scheme. For each $t\ge 0$ and $m\ge 1$, set
\begin{equation}\label{eq11}
u_{m}(t)=\displaystyle \left[T_{\frac{t}{m}}\right]^{m}u_0.
\end{equation}
In other words, $u_m(t)$ is the $m$-th term of the forward-backward sequence generated by \eqref{E:def_FBseq} from $u_0$ using the constant step size $\lambda_k\equiv t/m$. We shall prove that $(u_m)$ converges uniformly on compact intervals to a Lipschitz-continuous function satisfying \eqref{eq10}. We begin by establishing the convergence.

\begin{prop} \label{P:approximation}
The sequence $(u_m)$ converges pointwise on $[0,\infty)$, and uniformly on $[0,S]$ for each $S>0$, to a function $u:[0,\infty)\to X$, which is globally Lipschitz-continuous with constant $\nn(A+B)u_0\nn$. 
\end{prop}

\begin{proof} 		
We may assume that $u_0\in D(A)$. Extension to $\overline{D(A)}$ will then be possible in view of the Lipschitz (thus uniform) continuity. Given $t,s > 0$ and $n,m\in \N$, define $u_m(t)$ and $u_n(s)$ as above. By Theorem \ref{T:Kobayashi_A+B}, we have
\begin{equation} \label{A+B:koba_existence}
\|u_m(t)-u_n(s)\|\le \nn(A+B)u_0\nn\sqrt{(t-s)^2+\frac{t^2}{m}+\frac{s^2}{n}}.
\end{equation}
For $s=t$, this gives
$$\|u_m(t)-u_n(t)\|\le t\,\nn(A+B)u_0\nn\sqrt{\frac{1}{m}+\frac{1}{n}}.$$
It follows that $(u_m)$ converges pointwise on $[0,\infty)$, and uniformly on $[0,S]$ for each $S>0$, to a function $u:[0,\infty)\to X$. Passing to the limit in \eqref{A+B:koba_existence}, as $m,n\to\infty$, we obtain 
$$\|u(t)-u(s)\|\le\nn(A+B)u_0\nn\,|t-s|$$ 
for all $t,s>0$. 
\end{proof}

\begin{rem} \label{R:u_v}
Given $S>0$ and $m\ge 1$, define $v_m:[0,S]\to X$ by
\begin{equation}\label{eq11_v}
v_{m}(t)=\displaystyle \left[T_{\frac{S}{m}}\right]^{\mu(t)}u_0,\quad\hbox{where}\quad \mu(t)=\left\lfloor m\frac{t}{S}\right\rfloor\quad\hbox{and}\quad t\in [0,S].
\end{equation}
This is a piecewise constant interpolation of the forward-backward sequence generated with $\frac{S}{m}$ as step sizes, and initial point $u_0$ for $k=1,\dots m$. In order to estimate the distance between $v_m$ and $u_m$ (defined in \eqref{eq11}), we use \eqref{koba_A+B} to obtain
$$\|u_m(t)-v_m(t)\|\le \nn (A+B)u_0\nn \sqrt{\frac{S^2}{m^2}+\frac{t^2}{m}+\frac{tS}{m}} \le\frac{3S}{\sqrt{m}}\nn (A+B)u_0\nn.$$
Whence, as $m\to\infty$, $v_m$ also converges uniformly on $[0,S]$, for easch $S>0$, to the same function $u$. 
\end{rem}

\begin{thm}
The function $u$, given by Proposition \ref{P:approximation}, satisfies \eqref{eq10}.	
\end{thm}

\begin{proof}
We shall verify that $u$ is an integral solution of \eqref{eq10} in the sense of B\'enilan (see \cite{Ben}), which means that, whenever $y\in (A+B)x$ and $S\ge t>s \geq 0$, we have 
\begin{equation} \label{eq12}
	\|u(t)-x\|^{2}-\|u(s)-x\|^{2} \leq 2\int_{s}^{t}\langle j(x-u(\tau)),y\rangle d\tau.
\end{equation}

If $\left(x_{n}\right)$ is any sequence generated by (\ref{BF_iteration}) with steps sizes $\left( \lambda_{n}\right)$, then
$$-(x_n-x_{n-1})-\lambda_nBx_{n-1}+\lambda_nBx_n\in \lambda_nAx_n +\lambda_nBx_n
$$
for each $n\ge 1$. In view of the monotonicity of $A+B$, we have
$$\langle j(x-x_n), \lambda_ny+x_n-x_{n-1}+\lambda_nBx_{n-1}-\lambda_nBx_n\rangle \ge 0,$$
whenever $y\in Ax+Bx$. Whence,
\begin{eqnarray*}
2\lambda_n\langle j(x-x_n),y\rangle & \ge & 2\langle j(x-x_n),x_{n-1}-x_n\rangle+2\lambda_n\langle j(x-x_n),Bx_n-Bx_{n-1}\rangle \\
& = & 2\|x_n-x\|^2 +2\langle j(x-x_n),x_{n-1}-x\rangle +2\lambda_n\langle j(x-x_n),Bx_n-Bx_{n-1}\rangle \\
& \ge & \|x_n-x\|^2-\|x_{n-1}-x\|^2+2\lambda_n\langle j(x-x_n),Bx_n-Bx_{n-1}\rangle \\
& \ge & \|x_n-x\|^2-\|x_{n-1}-x\|^2 -2\theta^{-1}\lambda_n\|x-x_n\|\|x_n-x_{n-1}\|.
\end{eqnarray*}
Now, let us choose $x_0=u_0$, $\lambda_n\equiv \frac{S}{m}$, where $m$ is fixed but arbitrary. In view of Remark \ref{R:u_v},
there is a constant $K>0$ such that $2\theta^{-1}\|x-x_n\|\le K$ for $n=1,\dots,m$. Summing for $n=\mu(s),\cdots,\mu(t)$, we obtain 
 \begin{eqnarray*} \label{eq15}
	\|v_m(t)-x\|^{2}-\|u(s)-x\|^{2} & \leq &  2\sum_{n=\mu(s)}^{\mu(t)}\frac{S}{m}\big[\langle j(x-x_{n}),y\rangle+K\|x_n-x_{n-1}\|\big] \\
	& \leq &  2\sum_{n=\mu(s)}^{\mu(t)}\frac{S}{m}\langle j(x-x_{n}),y\rangle +\sum_{n=\mu(s)}^{\mu(t)}\frac{6S^2K\nn (A+B)u_0\nn}{m\sqrt{m}} \\
	& = &  2\sum_{n=\mu(s)}^{\mu(t)}\frac{S}{m}\langle j(x-x_{n}),y\rangle +(\mu(t)-\mu(s))\frac{6S^2K\nn (A+B)u_0\nn}{m\sqrt{m}} \\
	& \leq &  2\sum_{n=\mu(s)}^{\mu(t)}\frac{S}{m}\langle j(x-x_{n}),y\rangle +\frac{6S^2K\nn (A+B)u_0\nn}{\sqrt{m}}.
\end{eqnarray*}
We obtain \eqref{eq12} by letting $m\to\infty$.
\end{proof}
 
Existence of solution for \eqref{eq10} can be recovered as a consequence of the preceding arguments.

\begin{cor}
The differential inclusion \eqref{eq10} has a unique solution.	
\end{cor}

Uniqueness follows from monotonicity. Another consequence of the results above is:

\begin{cor} \label{C:B_nonincreasing} 
Let $(x_k)$ be a sequence generated by \eqref{E:def_FBseq} and let $u:[0,S]\rightarrow X$ be a solution of \eqref{eq10}. Then
\begin{enumerate}
	\item[(i)] The function $t\mapsto\nn(A+B)u(t)\nn$ is nonincreasing.
	\item[(ii)] $\|x_k-u(t)\|\leq \|x_0-u_0\|+\min\big\{\nn(A+B)x_0\nn,\nn(A+B)u_0\nn\big\}\sqrt{(\sigma_k-t)^{2}+\tau_k}$.
\end{enumerate}
\end{cor}

\section{Approximation in infinite horizon} \label{S:infinite_I}

In this section, we show that the forward-backward sequence generated by \eqref{E:def_FBseq}, have the same asymptotic behavior, as the number of iterations goes to infinity, as the solutions of the evolution equation \eqref{eq10}, when time does. The key argument is the idea of asymptotic equality introduced by Passty \cite{P}, closely related to the notion of almost-orbit, introduced by Miyadera and Kobayasi \cite{MK}. Further commentaries on this topic can be found in \cite{AlvPey1,AlvPey2,AlvPey3}.

In order to simplify the notation, given $x\in\overline{D(A)}$ and $t\ge 0$, we write
\begin{equation} \label{E:cal_S}
\mathcal S_tx=u(t),
\end{equation} 
where $u$ satisfies \eqref{eq10} with $u_0=x$. Also, for $0\le s\le t$, we write
\begin{equation} \label{E:U_S}
U_{\mathcal S}(t,s)=\mathcal S(t-s).
\end{equation} 
In a similar fashion, if $n\in\N$ and $x\in H$, we denote
\begin{equation} \label{E:cal_E}
\mathcal T_nx=T_{\lambda_n}\circ\cdots\circ T_{\lambda_1}x.
\end{equation} 
In other words, $\mathcal T_nx$ is the $n$-th term of the forward-backward sequence starting from $x\in\overline{D(A)}$. Assume $(\lambda_n)\notin\ell^1$, and write $\nu(t)=\max\{n\in\N:\sigma_n\le t\}$. For $0\le s\le t$, we set
\begin{equation} \label{E:U_E}
U_{\mathcal T}(t,s)=\prod_{i=\nu(s)+1}^{\nu(t)}T_{\lambda_i},
\end{equation} 
where the product denotes composition of functions and the empty composition is the identity.

A {\em nonexpansive evolution system} on $X$ is a family $\big(U(t,s)\big)_{0\le s\le t}$ such that
\begin{enumerate}
	\item[(i)] $U(t,t)z=z$ for all $z\in X$ and $t\ge 0$.
	\item[(ii)] $U(t,s)U(s,r)z=U(t,r)z$ for all $z\in X$ and all $t\ge s\ge r\ge 0$.
	\item[(iii)] $\|U(t,s)x-U(t,s)y\|{\leq}\|x-y\|$ for all $x,y\in X$ and $t\ge s\ge 0$.
\end{enumerate}

\begin{example} 
The families $\big(U_\mathcal S\big)$ and $\big(U_\mathcal T\big)$, defined in \eqref{E:U_S} and \eqref{E:U_E}, respectively, are nonexpansive evolution systems. Actually, the same is true if $\mathcal S$ is replaced by any other semigroup of nonexpansive functions on $X$, and if each $T_{\lambda_i}$ is replaced by any other nonexpansive function on $X$.
\end{example}

A function $\phi:[0,\infty)\to X$ is an {\it almost-orbit} of the nonexpansive evolution system $U$ if
$$\lim_{t\to\infty}\sup_{h\ge 0}\|\phi(t+h)-U(t+h,t)\phi(t)\|=0.$$
The following result from \cite[Theorem 3.3]{AlvPey2} reveals the usefulness of the concept of almost-orbit.

\begin{prop} \label{P:almost_orbit}
Let $U$ be a nonexpansive evolution system and let $\phi$ be an almost-orbit of $U$. If, for each $x\in X$ and $s\ge 0$, $U(t,s)x$ converges weakly (resp. strongly) as $t\to\infty$, then so does $\phi(t)$. The same holds if the word ``converges" is replaced by ``almost-converges" or ``converges in average".
\end{prop}

Several examples and applications, along with additional commentaries can be found in \cite{AlvPey2,ConPey1}.

The following result establishes a relationship between the trajectories generated by $U_{\mathcal S}$ and $U_{\mathcal T}$:

\begin{thm} \label{T:A+B_almost_orbits}
Let $(\lambda_n)\in\ell^2\setminus\ell^1$, and fix $x\in X$. For each $t>0$, define $\phi_{\mathcal S}(t)=\mathcal S_tx$ and $\phi_{\mathcal T}(t)=\mathcal T_{\nu(t)}x$\footnote{This is a piecewise constant interpolation of the sequence $\mathcal T_nx$.}. Then, $\phi_{\mathcal S}$ is an almost-orbit of $U_{\mathcal T}$, and
$\phi_{\mathcal T}$ is an almost-orbit of $U_{\mathcal S}$.	
\end{thm}

\begin{proof}
We first prove that $\phi_{\mathcal S}$ is an almost-orbit of $U_{\mathcal T}$. By Theorem \ref{T:Kobayashi_A+B} and Corollary \ref{C:B_nonincreasing}, we have
\begin{eqnarray*}
\left\|\left[\prod_{k=1}^mT_{\frac{h}{m}}\right]\mathcal S_tx-\left[\prod_{i=\nu(t)+1}^{\nu(t+h)}T_{\lambda_i}\right]\mathcal S_tx\right\| & \le & \nn(A+B)\mathcal S_tx\nn\sqrt{\left(\sigma_{\nu(t)+1}^{\nu(t+h)}-h\right)^{2}+\tau_{\nu(t)+1}^{\nu(t+h)}+\frac{h^{2}}{m}} \\
& \le & \nn(A+B)x\nn\sqrt{4\rho^{2}(t)+\tau_{\nu(t)+1}^{\infty}+\frac{h^{2}}{m}},
\end{eqnarray*}
where $\sigma_k^n=\sigma_n-\sigma_k$, $\tau_k^n=\tau_n-\tau_k$ and $\rho(t):=\sup\{\lambda_n:n\ge \nu(t)-1\}$, which vanishes as $t\to\infty$. Passing to the limit as $m\to \infty$, we obtain
$$ \|\mathcal S_h\mathcal S_tx-U_{\mathcal T}(t+h,t)\mathcal S_tx\|\le \nn(A+B)x\nn\sqrt{4\rho^{2}(t)+\tau_{\nu(t)+1}^{\infty}},
$$
which tends to $0$ as $t\to\infty$, uniformly in $h\ge 0$. It follows that
$$\lim_{t\to\infty}\sup_{h\ge 0}\|\phi_{\mathcal S}(t+h)-U_{\mathcal T}(t+h,t)\phi_{\mathcal S}(t)\|=0.$$
To prove that $\phi_{\mathcal T}$ is an almost-orbit of $U_{\mathcal S}$, we proceed in a similar fashion, to obtain
$$
\left\|\left[\prod_{i=\nu(t)+1}^{\nu(t+h)}T_{\lambda_i}\right]\mathcal T_{\nu(t)}x-\left[\prod_{k=1}^mT_{\frac{h}{m}}\right]\mathcal T_{\nu(t)}x\right\| \le \nn(A+B)x\nn\sqrt{4\rho^{2}(t)+\tau_{\nu(t)+1}^{\infty}+\frac{h^{2}}{m}}.
$$
Then, we pass to the limit as $m\to\infty$ to deduce that
$$\|\phi_{\mathcal T}(t+h)-\mathcal S_h\phi_{\mathcal T}(t)\| \le \nn(A+B)x\nn\sqrt{4\rho^{2}(t)+\tau_{\nu(t)+1}^{\infty}},$$
and conclude.
\end{proof}	

Theorem \ref{T:A+B_almost_orbits} implies \cite[Lemmas 4 \& 6]{P}, \cite[Proposition 2.3]{SugKoi}, \cite[Proposition 7.4]{MK}, \cite[Propositions 8.6 i) \& 8.7]{PeySor} and \cite[Theorem 3.1]{ConPey1}. Combining Theorem \ref{T:A+B_almost_orbits} with Proposition \ref{P:almost_orbit}, and using \cite[Lemma 5.3]{J}, we obtain

\begin{thm} \label{T:conclusion_infinite}
The following statements are equivalent:
\begin{itemize}
	\item [i)] For every $z\in \overline{D(A)}$, $\mathcal S_tz$ converges strongly (weakly), as $t\to+\infty$.
	\item [ii)] For every initial point $x_0\in X$, every sequence of step sizes $(\lambda_n)_{n\ge 1}\in\ell^2\setminus\ell^1$, and every sequence of errors $(\e_k)_{k\ge 1}$ such that $\sum_{k\ge 1}\|\e_k\|<+\infty$, the sequence $(x_n)$, generated by \eqref{E:FB_errors}, converges strongly (weakly), as $n\to+\infty$.
	\item [iii)] There exists a sequence of step sizes $(\lambda_n)_{n\ge 1}\in\ell^2\setminus\ell^1$ such that, for every initial point $x_0\in X$, the sequence $(x_n)$, generated by \eqref{BF_iteration}, converges strongly (weakly), as $n\to+\infty$.
\end{itemize}
\end{thm}

Theorem \ref{T:conclusion_infinite} implies \cite[Theorems 1 \& 2]{P}, \cite[Theorem]{SugKoi}, \cite[Theorem 7.5]{MK}, as well as \cite[Theorem 3.2]{ConPey1}.

\subsection*{New convergence results for forward backward sequences on Banach spaces}

Theorem \ref{T:conclusion_infinite} can automatically give new convergence results for forward-backward sequences by translating the information available on the behavior of the semigroup. Theorem \ref{T:FB_strong} below is provided as a {\it methodological} example, to show how this indirect analysis can be carried out. Therefore, we have priviledged statement simplicity, over generality.

Recall, from Section \ref{S:preliminaries}, that $X$ is a Banach space with 2-uniformly convex dual, $A$ is $m$-accretive and $B$ is cocoercive. Let $(\e_k)_{k\ge 1}$ be a sequence representing computational errors and let $(x_k)_{k\ge 0}$ satisfy \eqref{E:FB_errors}. We assume that $\sum_{k\ge 1}\|\e_k\|<+\infty$. Finally, set $\mathcal A=A+B$ and $\Sigma=\mathcal A^{-1}0$, and assume $\Sigma\neq\emptyset$. To simplify the statements and arguments, supose $X$ is uniformly convex. We know that $\Sigma$ is closed and convex, and the projection $P_\Sigma$ is well defined, single-valued and continuous.

\begin{thm} \label{T:FB_strong}
Let $(\lambda_n)_{n\ge 1}\in\ell^2\setminus\ell^1$. Assume one of the following conditions holds:
\begin{itemize}
	\item [i)] There is $\alpha>0$ such that for every $x\notin\Sigma$ and every $y\in \mathcal Ax$, $\langle j(x-P_\Sigma x),y\rangle\ge \alpha\|x-P_\Sigma x\|^2$;
	\item [ii)] $J_1$ is compact and, for every $x\notin\Sigma$ and every $y\in \mathcal Ax$, 
	$\langle j(x-P_\Sigma x),y\rangle>0$; or
	\item [iii)] The interior of $\Sigma$ is not empty.
\end{itemize}
Then, $x_n$ converges strongly, as $n\to+\infty$, to a point in $\Sigma$.
\end{thm}

\begin{proof}
In all three cases, we first prove that for each $z\in\overline{D(A)}$, $\mathcal S_tz$ converges strongly, as $t\to+\infty$, to a point in $\Sigma$. 
\begin{itemize}
	\item [i)] The hypotheses of \cite[Theorem 1]{NevRei} are easily verified.
	\item [ii)] It suffices to combine \cite[Proposition 1]{NevRei} and \cite[Theorem 1]{NevRei}.
	\item [iii)] We use \cite[Theorem 4]{NevRei}.
\end{itemize}
We conclude by applying Theorem \ref{T:conclusion_infinite}.
\end{proof}

\section{Proof of the fundamental inequality} \label{S:proof}

This last section is devoted to the proof of Theorem \ref{T:Kobayashi_A+B}. In order to simplify the notation, given $\nu>0$ and $z,d\in X$, write
$$E_\lambda^\e(z)=E_\lambda(z)+\lambda \e,\quad\hbox{and}\quad T_\lambda^\e(z)=J_\lambda(E_\lambda^\e(z)),$$
so that \eqref{E:def_FBseq_errors} reads 
$$x_{k}=T_{\lambda_k}^{\e_k}(x_{k-1}).$$
Next, given $\Theta>0$ and $\lambda,\mu\in(0,\Theta]$, set
\begin{equation}\label{eq3}
\alpha=\frac{\lambda(\Theta-\mu)}{\Theta(\lambda+\mu)-\lambda\mu}, \;
\beta=\frac{\mu(\Theta-\lambda)}{\Theta(\lambda+\mu)-\lambda\mu},\;\gamma=\frac{\lambda\mu}{\Theta(\lambda+\mu)-\lambda\mu}.
\end{equation}

\begin{lem}\label{T_lambda-T_mu}
	Write $\Theta=\frac{\theta}{\kappa}$. For $\lambda,\mu\in (0,\Theta]$ and $x,y,\e,\eta\in X$, we have
	\begin{equation}\label{eq2}
	\|T_{\lambda}^{\e}(x)-T_{\mu}^{\eta}(y)\|\leq \alpha\|T_{\lambda}^{\e}(x)-y\|+\beta\|x-T_{\mu}^{\eta}(y)\|+\gamma\|x-y\|+\gamma\Theta\|\e-\eta\|.
	\end{equation}
\end{lem}

\begin{proof}
	Set $\Delta=j(T_{\lambda}^\e(x)-T_{\mu}^\eta(y))$. We have 
	\begin{eqnarray} \label{T_lam-T_mu-Banach}
	\Theta(\lambda+\mu)\|T_{\lambda}^\e(x)-T_{\mu}^\eta(y)\|^{2}&=&\Theta(\lambda+\mu)\langle T_{\lambda}^\e(x)-T_{\mu}^\eta(y),\Delta\rangle \nonumber\\
	& = & \Theta\lambda\langle T_{\lambda}^\e(x)-E_{\mu}^\eta(y),\Delta \rangle +\Theta\mu\langle E_{\lambda}^\e(x)-T_{\mu}^\eta(y),\Delta\rangle \nonumber\\
	&& +\Theta\lambda\mu\left\langle\frac{E_{\mu}^\eta(y)-T_{\mu}^\eta(y)}{\mu}-\frac{E_{\lambda}^\e(x)-T_{\lambda}^\e(x)}{\lambda},\Delta\right\rangle \nonumber\\
	& \leq & \Theta\lambda\langle T_{\lambda}^\e(x)-E_{\mu}^\eta(y),\Delta \rangle +\Theta\mu\langle E_{\lambda}^\e(x)-T_{\mu}^\eta(y),\Delta\rangle,
	\end{eqnarray}
	since $A$ is accretive and
	$$\frac{E_{\nu}^\e(z)-T_{\nu}^\e(z)}{\nu}\in A(T_{\nu}^\e(z))$$
	for all $\nu>0$ and $z,\e\in X$. We can rewrite \eqref{T_lam-T_mu-Banach} as
	\begin{equation}\label{T_lam-T_mu-Banach_bis}
	\Theta(\lambda+\mu)\|T_{\lambda}^\e(x)-T_{\mu}^\eta(y)\|^{2} \le \Theta\lambda\langle T_{\lambda}^\e(x)-y,\Delta \rangle +\Theta\mu\langle x-T_{\mu}^\eta(y),\Delta\rangle- \lambda\Theta\mu \langle Bx-\e-By+\eta,\Delta\rangle.
	\end{equation}
	Notice also that 
	\begin{equation}\label{eq4}
	-\lambda\mu\|T_{\lambda}^\e(x)-T_{\mu}^\eta(y)\|^{2} = -\lambda\mu\langle T_{\lambda}^\e(x)-y,\Delta\rangle -\lambda\mu\langle x-T_{\mu}^\eta(y),\Delta\rangle + \lambda\mu\langle x-y,\Delta\rangle.
	\end{equation}
	Combining (\ref{T_lam-T_mu-Banach_bis}) and (\ref{eq4}), we obtain	
	\begin{eqnarray*}
		[\Theta(\lambda+\mu)-\lambda\mu]\|T_{\lambda}^\e(x)-T_{\mu}^\eta(y)\|^{2} 
		& \leq & \lambda(\Theta-\mu)\langle T_{\lambda}^\e(x)-y,\Delta\rangle + \mu(\Theta-\lambda)\langle x-T_{\mu}^\eta(y),\Delta\rangle \\
		&  & +\ \lambda\mu\langle E_{\Theta}(x)-E_{\Theta}(y),\Delta\rangle +\lambda\mu\Theta\langle \e-\eta,\Delta\rangle.
	\end{eqnarray*}
	Since $E_\Theta$ is nonexpansive and $\|\Delta\|=\|T_{\lambda}^\e(x)-T_{\mu}^\eta(y)\|$, we finally get \eqref{eq2}.
\end{proof}

We are now in a position to conclude.

\begin{prop}
	Theorem \ref{T:Kobayashi_A+B} is true.
\end{prop}

\begin{proof} 
	To simplify notation set $$c_{k,l}=\sqrt{(\sigma_{k}-\hat{\sigma}_{l})^{2}+\tau_{k}+\hat{\tau}_{l}}.$$ 
	In view of the characterization \eqref{E:def_FBseq_errors} of the sequence $(x_k)$, for each $k\ge 1$, we have
	$$y_k:=\frac{E_{\lambda_k}(x_{k-1})+\lambda_k\e_k-x_k}{\lambda_k}\in Ax_k.$$
	Given any $v\in Au$, the accretivity of $A$ implies
	\begin{eqnarray*} 
		\|x_k-u\|\le\|x_k+\lambda_ky_k-u-\lambda v\| 
		& =& \|E_{\lambda_k}(x_{k-1})-E_{\lambda_k}(u)-\lambda_k(v+Bu)+\lambda_k\e_k\| \\ 
		& \le & \|E_{\lambda_k}(x_{k-1})-E_{\lambda_k}(u)\|+\lambda_k\|(v+Bu)\|+\lambda_k\|\e_k\|.
	\end{eqnarray*}
	Since $E_{\lambda_k}$ is nonexpansive and $v\in Au$ is arbitrary, we deduce that
	$$\|x_k-u\|\le\|x_{k-1}-u\|+\lambda_k\nn(A+B)u\nn+\lambda_k\|\e_k\|.$$
	Iterating this inequality, we obtain
	$$\|x_k-u\|\le\|x_{0}-u\|+\sigma_k\nn(A+B)u\nn+e_k,$$	
	and, noticing that $\sigma_{k}\leq c_{k,0}$, we conclude that
	$$\|x_{k}-\hat{x}_{0}\|\le \|x_{0}-u\|+\|\hat{x}_{0}-u\|+ c_{k,0}\nn(A+B)u\nn+\lambda_k\|\e_k\|,$$
	thus inequality \eqref{koba_A+B} holds for the pair $(k,0)$. For $(0,l)$, with $l\geq 0$, the argument is analogous. 
	
	The proof will continue using induction on the pair $(k,l)$. Let us assume inequality \eqref{koba_A+B} holds for the pairs $(k-1,l-1)$, $(k,l-1)$ and $(k-1,l)$, and show that it also holds for the pair $(k,l)$. To this end, we use the inequality (\ref{eq2}) with $x=x_{k-1}$, $y=\hat x_{l-1}$, $\lambda=\lambda_k$ and $\mu=\hat\lambda_l$:
	\begin{equation}\label{eq9}
	\|x_{k}-\hat{x}_{l}\|\leq \alpha_{k,l}\|x_{k}-\hat{x}_{l-1}\|+\beta_{k,l}\|x_{k-1}-\hat{x}_{l}\|+\gamma_{k,l}\|x_{k-1}-\hat{x}_{l-1}\|+\gamma_{k,l}\Theta\|\e_k-\hat\e_l\|.
	\end{equation}
	Using the induction hypothesis in \eqref{eq9} and the fact that $\alpha_{k,l}+\beta_{k,l}+\gamma_{k,l}=1$, we deduce that
	\begin{eqnarray} \label{A+B:koba_intermediate}
	\|x_{k}-\hat{x}_{l}\|&\le
	&\|x_0-u\|+\|\hat x_0-u\| +\nn(A+B)u\nn\left(\alpha_{k,l}c_{k,l-1}+\beta_{k,l}c_{k-1,l}+\gamma_{k,l}c_{k-1,l-1}\right) \nonumber \\
	&& +\ \alpha_{k,l}(e_k+\hat e_{l-1}) + \beta_{k,l}(e_{k-1}+\hat e_{l})+\gamma_{k,l}(e_{k-1}+\hat e_{l-1})+\gamma_{k,l}\Theta(\|\e_k\|+\|\hat\e_l\|)\nonumber \\
	& = & \|x_0-u\|+\|\hat x_0-u\| +\nn(A+B)u\nn\left(\alpha_{k,l}c_{k,l-1}+\beta_{k,l}c_{k-1,l}+\gamma_{k,l}c_{k-1,l-1}\right) \nonumber \\
	&& +\ e_{k-1}+\hat e_{l-1}+(\alpha_{k,l}\lambda_k+\gamma_{k,l}\Theta)\|\e_k\| +(\beta_{k,l}\hat\lambda_l+\gamma_{k,l}\Theta)\|\hat\e_l\| \nonumber \\
	& = & \|x_0-u\|+\|\hat x_0-u\| +\nn(A+B)u\nn\left(\alpha_{k,l}c_{k,l-1}+\beta_{k,l}c_{k-1,l}+\gamma_{k,l}c_{k-1,l-1}\right) + e_k+\hat e_l, 
	\end{eqnarray} 
	since $\alpha_{k,l}\lambda_k+\gamma_{k,l}\Theta=\lambda_k$ and $\beta_{k,l}\hat\lambda_l+\gamma_{k,l}\Theta=\hat\lambda_l$. On the other hand, we have
	\begin{eqnarray} \label{A+B:koba_CS}
	\alpha_{k,l}c_{k,l-1}+\beta_{k,l}c_{k-1,l}+\gamma_{k,l}c_{k-1,l-1}& \le & \sqrt{\alpha_{k,l}+\beta_{k,l}+\gamma_{k,l}}\sqrt{\alpha_{k,l}c_{k,l-1}^2+\beta_{k,l}c_{k-1,l}^2+\gamma_{k,l}c_{k-1,l-1}^2} \nonumber \\
	&=& \sqrt{\alpha_{k,l}c_{k,l-1}^2+\beta_{k,l}c_{k-1,l}^2+\gamma_{k,l}c_{k-1,l-1}^2},
	\end{eqnarray}
	and
	\begin{eqnarray*}
		c^{2}_{k,l-1} & = & c^{2}_{k,l}+2\hat{\lambda}_{l}(\sigma_{k}-\hat{\sigma}_{l}) \\
		c^{2}_{k-1,l} & = & c^{2}_{k,l}+2\lambda_{k}(\sigma_{k}-\hat{\sigma}_{l}) \\ 
		c^{2}_{k-1,l-1} & = & c^{2}_{k,l}+2(\hat{\lambda}_{l}-\lambda_{k})(\sigma_{k}-\hat{\sigma}_{l})-2{\lambda_{k}}{\hat{\lambda}_{l}}. 
	\end{eqnarray*}
	Therefore,
	\begin{eqnarray} \label{A+B:koba_ckl}
	\alpha_{k,l}c_{k,l-1}^2+\beta_{k,l}c_{k-1,l}^2+\gamma_{k,l}c_{k-1,l-1}^2=c_{k,l}^2-2\gamma_{k,l}{\lambda_{k}}{\hat{\lambda}_{l}}\le c_{k,l}^2.
	\end{eqnarray} 
	Combining \eqref{A+B:koba_intermediate}, \eqref{A+B:koba_CS} and \eqref{A+B:koba_ckl}, we obtain
	\eqref{koba_A+B}.
\end{proof}

\end{document}